\documentclass[12pt, twoside]{article}
\usepackage{amssymb,amsmath,eufrak,amscd}
\usepackage[dvips]{graphicx}
\newcommand{\Ker}[1]{\mathrm{Ker}\left(#1\right)}     %  Kernel
\newcommand{\abs}[1]{\left\vert#1\right\vert}         %  valore assoluto
\newcommand{\set}[1]{\left\{#1\right\}}               %  insieme
\newcommand{\seq}[1]{\left<#1\right>}                 %  sottogr generato
\newcommand{\lto}{\longrightarrow}

\newcommand{\lmto}{\longmapsto}
\newcommand{\rk}{\mathrm{rk}}                  %   rango
\newcommand{\mm}[1]{\mathbb#1}                 %  mathbb
\newtheorem{theorem}{Theorem}[section]
\newtheorem{prop}{Proposition}

\newtheorem{lem}[theorem]{Lemma}
\newtheorem{cor}[theorem]{Corollary}
\newtheorem{pro}[theorem]{Proposition}
\newtheorem{que}{Question}%[section]
\newtheorem{rem}{Remark}%[section]

\input{xy}
\xyoption{all}

\begin{document}
\title{Some structural results on the non-abelian tensor square of groups}

\author{Russell D. Blyth \and Francesco Fumagalli \and Marta Morigi}

\date{}
\maketitle

\begin{abstract}
We study the non-abelian tensor square $G\otimes G$ for the class
of groups $G$ that are finitely
generated  modulo their derived subgroup. In particular, we find
conditions on $G/G'$ so that
$G\otimes G$ is isomorphic to the direct product of 
$\nabla(G)$ and the non-abelian exterior
square $G\wedge G$. For any group $G$, we characterize the non-abelian exterior
square $G\wedge G$ in terms of a presentation of $G$.
Finally, we apply our results to some classes of groups, such as
the classes of free soluble and free nilpotent groups of finite
rank, and some classes of finite $p$-groups.\\

\end{abstract}
\noindent
\section*{Introduction}
The non-abelian tensor square $G\otimes G$ of a group $G$ is a
special case of the non-abelian tensor product $G\otimes H$ of two
arbitrary groups $G$ and $H$, that was introduced by Brown and
Loday in \cite{BL1,BL2} and arises from applications of a generalized Van Kampen theorem
 in homotopy
theory.

For all $g,h\in G$ let $\null^gh=ghg^{-1}$ and
$[g,h]=ghg^{-1}h^{-1}$. Then $G\otimes G$ is defined as the group
generated by the symbols $g\otimes h$, for $g,h\in G$, subject to
the relations
$$gh\otimes k=({^gh}\otimes{^gk})(g\otimes k)\qquad\mathrm{and}\qquad
g\otimes hk=(g\otimes h)({^hg}\otimes{^hk}).$$
The definition guarantees the existence of an epimorphism
$\kappa:G\otimes G\lto G'$, defined on the generators by
$\kappa(g\otimes h)=[g,h]$ for all $g, h\in G$. Let $J(G)$ be the
kernel of the map $\kappa$, and let $\nabla(G)$ be the normal
subgroup generated by the elements $g\otimes g$, for all $g\in G$.
The group $(G\otimes G)/\nabla(G)$ is called \emph{the non-abelian
exterior square} of $G$, and denoted by $G\wedge G$. The map
$\kappa$ factorizes modulo $\nabla(G)$, thus inducing an
epimorphism $\kappa':G\wedge G\lto G'$. By results in
\cite{BL1,BL2} the kernel of the map $\kappa'$ is isomorphic to
the Schur multiplicator $M(G)$ of $G$. Let $\Gamma(G/G')$ be
Whitehead's quadratic functor, as defined in \cite{Whi}. Then
results in \cite{BL1,BL2} give a commutative diagram with exact
rows and central extensions as columns:
 $$\xymatrix{ & & 1\ar[d] & 1\ar[d]& \\
         &\Gamma(G/G')\ar[d]\ar[r]&J(G)\ar[d]\ar[r]&M(G)\ar[d]\ar[r]&1\\
 1\ar[r]&\nabla(G)\ar[d]\ar[r]&G\otimes G\ar[d]_{\kappa}\ar[r]&G\wedge G\ar[d]_{\kappa'}\ar[r]&1\\
  &1\ar[r]&G' \ar[r]^{\mathrm{id}}\ar[d]&G'\ar[d]&\\ & & 1&1&}$$

We are interested in the case when the middle  row of the above
diagram splits. Our main result in this context is the following.

\begin{prop}\label{spe}
Let $G$ be a group such that $G/G'$ is finitely generated. If
$G/G'$ has no elements of order two or if $G'$ has a complement in
$G$ then $G\otimes G\simeq \nabla(G)\times (G\wedge G)$.
\end{prop}

We will see that, under the hypotheses of Proposition 1, the
structure of the tensor square $G\otimes G$ is completely
determined once the structures of $G/G'$ and of $G\wedge G$ are
known. In \cite{BJR} Brown, Johnson and Robertson proved that if
$M(G)$ is finitely generated then $G\wedge G$ is isomorphic to the
derived subgroup of any covering group $\hat G$ of $G$ (the notion
of a covering group is well known if $G$ is finite, see
\cite{Karpi}, and in the general case
the authors of \cite{BJR} adopted a similar definition).\\
Our contribution is the following.

\begin{prop}\label{ex}
Let $G$ be a group and let $F$  be a free group such that $G\simeq
F/R$ for some normal subgroup $R$ of $F$. Then
$$G\wedge G\simeq\frac{F'}{[F,R]}.$$
\end{prop}

As corollaries of Propositions \ref{spe} and \ref{ex}, we deduce the
structures of non-abelian tensor squares of finitely generated groups that are
free in some variety, for example, the free $n$-generated
nilpotent groups of fixed class (see Corollary
\ref{freenilpotent}) or the free $n$-generated soluble groups of
fixed derived length (see Corollary \ref{freesolvable}).\\

We mention here that a wide list of references on the non-abelian
tensor square of a group $G$ can be found in \cite{Mo} and that an
effective algorithm for computing it in the case when $G$ is
polycyclic has been recently developed by Eick and Nickel in
\cite{EN} and is implemented in the computing program GAP.\\

The paper is organized as follows. In the first section we collect
some background material and prove some new basic results on the
tensor square of an arbitrary group $G$. Proposition \ref{spe} is
proved in Section 2, while in Section 3 we prove Proposition
\ref{ex} and derive several consequences. Section 4 deals with
finite $p$-groups $G$; in particular some upper bounds on the
orders of
$G\otimes G$ and $M(G)$ are found.\\

The notation used in this paper is standard (the reader is
referred for example to \cite{Karpi}), with the only exception
that conjugation and commutation are as defined in the second
paragraph of this Introduction.

\section{Background}
Let $G$ be an arbitrary group. In order to investigate the
structure of $G\otimes G$, it is sometimes more convenient to
consider the following construction, which was introduced in
\cite{EllisLeo}.\\
Let $G^\varphi$ be a group isomorphic to $G$ via the isomorphism
$\varphi:G\lto G^\varphi$, and consider the group
$$\nu(G):=\seq{G,G^\varphi\big|{\cal R}, {\cal R}^\varphi, \null^{g_3}[g_1,g_2^\varphi]=
[\null^{g_3}g_1,(\null^{g_3}g_2)^\varphi]=\null^{g_3^\varphi}[g_1,g_2^\varphi],
\forall g_1,g_2,g_3\in G},$$ where ${\cal R}, \cal{R}^\varphi$ are
the defining relations of $G$ and $G^\varphi$ respectively (that
is, $\nu(G)$ is the quotient of the free product $G\ast G^\varphi$
by its normal subgroup generated by all the words
$\null^{g_3}[g_1,g_2^\varphi]\cdot
[\null^{g_3}g_1,(\null^{g_3}g_2)^\varphi]^{-1}$ and
$\null^{g_3^\varphi}[g_1,g_2^\varphi]\cdot
[\null^{g_3}g_1,(\null^{g_3}g_2)^\varphi]^{-1}$, $g_1,g_2,g_3\in
G$).
In \cite{Rocco91} (Proposition 2.6), the non-abelian tensor square
$G\otimes G$ is proved to be isomorphic to the commutator subgroup
$[G,G^\varphi]$ inside $\nu(G)$.\\
From now on we identify $G\otimes G$ with $[G,G^\varphi]$ and,
unless differently specified, we write $[g,h^\varphi]$ in place of
$g\otimes h$ (for $g,h\in G$). For the reader's clarity we report
here some results that we will often use.
\begin{lem}[Lemma 2.1 in \cite{Rocco91}, Lemma 2.1 in
\cite{BMMfree}]\label{basiccomm}
%non-abelian tensor square $G\otimes G$
Let $G$ be any group. The following relations hold in $\nu(G)$.
\begin{enumerate}
\item[(i)]
$\null^{[g_3,g_4^\varphi]}[g_1,g_2^\varphi]=\null^{[g_3,g_4]}[g_1,g_2^\varphi]=\null^{[g_3^\varphi,g_4]}[g_1,g_2^\varphi]$,
for all $g_1, g_2, g_3, g_4\in G$.
\item[(ii)] $[g_1^\varphi,g_2,g_3]=[g_1,g_2^\varphi,g_3]=[g_1,g_2,g_3^\varphi]=
           [g_1^\varphi,g_2^\varphi,g_3]=[g_1^\varphi,g_2,g_3^\varphi]=[g_1,g_2^\varphi,g_3^\varphi]$,
          for all $g_1,g_2, g_3\in G$.
\item[(iii)] $[g_1,[g_2,g_3]^\varphi]=[g_2,g_3,g_1^\varphi]^{-1}$, for all $g_1,g_2, g_3\in G$.
\item[(iv)] $[g,g^\varphi]$ is central in $\nu(G)$ for all $g\in G$.
\item[(v)] $[g_1,g_2^\varphi][g_2,g_1^\varphi]$ is central in $\nu(G)$ for all $g_1,g_2\in G$.
\item[(vi)] $[g,g^\varphi]=1$ for all $g\in G'$.
\end{enumerate}
\end{lem}
\begin{cor}\label{basicconse}
Let $G$ be any group. Then the following hold.
\begin{enumerate}
\item[(i)] If $a,b\in G$ commute, then $[a,b^\varphi]$ and
$[b,a^\varphi]$ are central elements of $[G,G^\varphi]$.
\item[(ii)] If $g_1\in G'$ or $g_2\in G'$, then
$[g_1,g_2^\varphi]^{-1}=[g_2,g_1^\varphi]$.
\item[(iii)] If $A$ and $B$ are two subgroups of $G$ with $B\leq G'$, then
$[A,B^\varphi]=[B,A^\varphi]$. In particular, $[G,G'^\varphi]=[G',G^\varphi]$.
\item[(iv)] $[G',Z(G)^\varphi]=1$.
\end{enumerate}
\end{cor}
\begin{proof}
(i) By Lemma \ref{basiccomm} (ii) it follows that both
$[a,b^\varphi]$ and $[b,a^\varphi]$ commute with any element of
$G$ and $G^\varphi$, so they are
indeed central elements of $\nu(G)$.\\
(ii) By Lemma \ref{basiccomm} (iii) the result holds if either $g_1$ or
$g_2$ is a commutator. An easy calculation shows that the result
holds if $g_1$ or $g_2$ are arbitrary elements of $G'$. (See also \cite{Rocco94}, Lemma 3.1 (iii)).\\
(iii) Assume that $A=\seq{a_i\vert i\in I}$ and $B=\seq{b_j\vert
j\in J}\leq G'$. Then $[A,B^\varphi]$ is generated by
$[a_i,b_j^\varphi]$ ($i\in I, j\in J$), which by (ii) is equal to
$[b_j,a_i^\varphi]^{-1}$. Hence $[A,B^\varphi]\leq [B,A^\varphi]$,
and, by a symmetric argument,
we have $[A,B^\varphi]=[B,A^\varphi]$.\\
(iv) See \cite{Rocco91}, Lemma 2.7.
\end{proof}
These results permit a description to be provided for the derived and the
lower central series of $G\otimes G$ in terms of those of $G$.
\begin{pro}\label{serterms}
Let $G$ be any group. Then the following hold:
\begin{enumerate}
\item[(i)] For every $n\geq 0$, $[G,G^\varphi]^{(n)}=[G^{(n)},(G^{(n)})^\varphi]$.
\item[(ii)] For every $n\geq 1$, $\gamma_{n+1}([G,G^\varphi])=[\gamma_{n}(G'),G'^{\varphi}]=[G',\gamma_{n}(G')^{\varphi}]$.
\end{enumerate}
\end{pro}
\begin{proof}$\null$\\
(i) We use induction on $n$. The result being trivial for $n=0$,
assume $n=1$. For $g_i\in G$ ($i=1,\ldots,4$) we have the following identity (which is stated without proof in Lemma 11 of \cite{BMpoly}):
\begin{align*}
(\ast)\, \,  [g_1,g_2^{\varphi},[g_3,g_4^{\varphi}]]&=
[g_1,g_2^{\varphi}]\big(\null^{[g_3,g_4^{\varphi}]}[g_1,g_2^{\varphi}]\big)^{-1}\\
&=[g_1,g_2^{\varphi}]\big(\null^{[g_3,g_4]}[g_1,g_2^{\varphi}]\big)^{-1}\quad {\textrm{ by Lemma \ref{basiccomm}(i),}}\\
&=[g_1,g_2^{\varphi},[g_3,g_4]^{\varphi}]\quad {\textrm{ by the defining properties of $\nu(G)$,}}\\
&=[g_1,g_2,[g_3,g_4]^{\varphi}]\quad {\textrm{ by Lemma
\ref{basiccomm}(ii).}}
\end{align*}
We note that both $[G,G^{\varphi}]'$ and $[G',G'^\varphi]$ are normal in $\nu(G)$. By 5.1.7 of \cite{Rob} we have that $[G,G^{\varphi}]'=[G,G^{\varphi},[G,G^{\varphi}]]$ is the normal closure in $\nu(G)$ of the subgroup
generated by the elements of the form
$$[g_1,g_2^{\varphi},[g_3,g_4^{\varphi}]],$$
and $[G',G'^\varphi]$ is the normal closure in $\nu(G)$ of the
subgroup generated by the elements of the form
$$[g_1,g_2,[g_3,g_4]^{\varphi}].$$
Therefore $(\ast)$ shows that $[G,G^\varphi]'=[G',(G')^\varphi]$.\\
We now assume that the result is true for $n$ and we prove it for $n+1$.\\
By the inductive hypothesis and the argument above applied
to the group $G^{(n)}$, we have that
$$[G,G^{\varphi}]^{(n+1)}=([G,G^{\varphi}]^{(n)})'=([G^{(n)},(G^{(n)})^{\varphi}])'=[G^{(n+1)},(G^{(n+1)})^{\varphi}].$$

\noindent (ii) The case $n=1$ has already been proved in (i). Thus
we assume that the result is true for $n$ and we prove it for
$n+1$. By the inductive hypothesis, we have that
$$\gamma_{n+2}([G,G^{\varphi}])=
[\gamma_{n+1}([G,G^{\varphi}]),[G,G^{\varphi}]]=
[\gamma_{n}(G'),G'^{\varphi},[G,G^{\varphi}]].$$
By 5.1.7 of \cite{Rob} $[\gamma_{n}(G'),G'^{\varphi},[G,G^{\varphi}]]$ is the normal closure in $\nu(G)$ of
the group generated by the elements of the
form
$$[g_1,g_2^{\varphi},[g_3,g_4^{\varphi}]],\quad {\textrm{ with }}g_1\in \gamma_{n}(G'),
g_2\in G' {\textrm{ and }} g_3,g_4\in G.$$

Similarly, $[\gamma_{n+1}(G'),G'^{\varphi}]$ is the normal closure in $\nu(G)$ of
the group generated by the elements of the
form
$$[[g_1,g_2],[g_3,g_4]^{\varphi}],\quad {\textrm{ with }}g_1\in \gamma_{n}(G'),
g_2\in G' {\textrm{ and }} g_3,g_4\in G.$$
By $(\ast)$, we have that
$[g_1,g_2^{\varphi},[g_3,g_4^{\varphi}]]=[[g_1,g_2],[g_3,g_4]^{\varphi}]$, so that
$\gamma_{n+2}([G,G^{\varphi}])=[\gamma_{n+1}(G'),G'^{\varphi}]$. This completes the induction.
Finally, by Corollary \ref{basicconse} (iii), we have
$$[[\gamma_{n+1}(G'),G'^{\varphi}]=[G',\gamma_{n+1}(G')^{\varphi}],$$
for all $n\geq 1$.
\end{proof}
We stress the fact that Proposition \ref{serterms} does not say
that in general $(G\otimes G)^{(n)}$ and $G^{(n)}\otimes G^{(n)}$
are isomorphic groups. Consider, for example, the
case $G=S_3$, where $G\otimes G$ is elementary abelian of order 4,
while $A_3\otimes A_3$ has order order 3). Indeed, the computation of $(G\otimes G)^{(n)}$ as
$[G,G^\varphi]^{(n)}=[G^{(n)},(G^{(n)})^\varphi]$
occurs within the group $\nu(G)$, whereas the calculation of $G^{(n)}\otimes G^{(n)}$ occurs as 
$[G^{(n)},(G^{(n)})^\varphi]$ within the group
$\nu(G^{(n)})$.
%Indeed, the notation
%used here is as ambiguous as the one of ordinary tensor products
%of modules (see \cite{AtMcD}, Chapter 2).\\

\noindent The following facts (given in \cite{BJR}) are seen to be
consequences of Proposition \ref{serterms} and Lemma
\ref{basicconse}.

\begin{cor}
If $G$ is a solvable group of derived length $d$, then $G\otimes G$ is solvable of
derived length $d-1$ or $d$.\\
If $G$ is a nilpotent group of class $c$, then $G\otimes G$ is
nilpotent of class $\leq \lfloor\frac{c+1}{2}\rfloor$.
\end{cor}

\section{The structure of the non-abelian tensor square}
In this section we prove some fundamental facts about the structure
of the non-abelian tensor square of any group $G$ such that
$G^{ab}=G/G'$ is finitely generated. We conjecture that our results
remain true under the weaker assumption that $G/G'$ is a
restricted direct product of cyclic groups.\\

\medskip

\noindent For a finitely generated abelian group $A$, its non-abelian tensor
square is simply the ordinary tensor product of two copies of $A$.
In particular, if ${\cal A}=\{a_1,\ldots, a_s\}$ is a set of generators of $A$ such that $A$ is the direct product of the cyclic groups
$\seq{a_i}$, $i=1,\ldots,s$, then we can write
$$A\otimes A=\nabla(A)\times E_{\cal A}(A),$$
where
\begin{align*}
\nabla(A)&=\seq{[a_i,a_i^\varphi],
[a_i,a_j^\varphi][a_j,a_i^\varphi]\vert 1\leq i<j\leq s}
\quad {\mathrm{and}}\\
     E_{\cal A}(A)&=\seq{[a_i,a_j^\varphi]\vert 1\leq i<j\leq s}.
\end{align*}

We observe that $\nabla(A)$ is independent of the set of
generators ${\cal A}$ of $A$, since in fact
$\nabla(A)=\seq{[a,a^\varphi]\vert a\in A}$, while $E_{\cal A}(A)$
is a complement of $\nabla(A)$ in $A\otimes A$ that does depend
on the choice of ${\cal A}$.

It turns out that for any group $G$ such that $G^{ab}$ is
finitely generated (in particular, for any finitely generated group
$G$), the structure of $\nabla(G)$ essentially depends  on $G^{ab}$.
The following Lemma, which improves Proposition 3.3 of Rocco
\cite{Rocco94}, makes this observation precise.
\begin{lem}\label{gene}
Let $G$ be a group such that $G^{ab}$ is finitely generated by the
elements $\set{x_iG'\vert i=1,\ldots,s}$. Set $E(G)$ to be 
$\seq{[x_i,x_j^\varphi]\vert i< j}[G', G^\varphi]$.
Then the following hold:
\begin{enumerate}
\item[(i)] $\nabla(G)$ is generated by the elements of the set\\
$\set{[x_i,x_i^\varphi], [x_i,x_j^\varphi][x_j,x_i^\varphi]\vert
1\le i< j\le s}$.
\item[(ii)] $[G, G^\varphi]=\nabla(G)E(G)$.
\end{enumerate}
\end{lem}
\begin{proof}$\null$\\
(i) Let $\textrm{Y}=\set{y_{\alpha}}_{\alpha\in I}$ be a set of
generators for $G'$ and let $\textrm{X}=\set{x_i}_{i=1}^s$. Then
$\mathcal{G}=\textrm{X}\cup\textrm{Y}$ generates $G$. By Lemma 17
in \cite{BMpoly} (or Proposition 3.3 in \cite{Rocco94})
$\nabla(G)$ is generated by
$$\set{[a,a^\varphi], [a,b^\varphi][b,a^\varphi]\vert a,b\in
\mathcal{G}}.$$
Note that $[a,a^\varphi]=1$ if $a\in\textrm{Y}$ (by Lemma
\ref{basiccomm}(vi)) and similarly $[a,b^\varphi][b,a^\varphi]=1$
if at least one among $a$ and $b$ lies in $\textrm{Y}$ (Corollary
\ref{basicconse} (ii)).

\noindent (ii)  The proof follows by a direct expansion of
the factors $[x_ig_1,(x_jg_2)^\varphi]$ ($g_1,g_2\in G'$). Alternatively, 
consider the map $f:[G, G^\varphi]\lto [G^{ab}, (G^{ab})^\varphi]$
induced by the projection onto $G^{ab}$. Then
$\mathrm{Im}\,f=f\big(\nabla(G)\seq{[x_i,x_j^\varphi]\vert i<
j}\big)$ and $\mathrm{Ker}\,f=[G', G^\varphi]=[G, (G')^\varphi]$ (see
\cite{Rocco91}, Remark 3), so $[G, G^\varphi]=\nabla(G)E(G)$.
\end{proof}

We are now able to describe the structure of the non-abelian tensor square
$G\otimes G$ in terms of  $\nabla(G)$ and the non-abelian exterior
square $G\wedge G$. Our result generalizes
Proposition 8 in \cite{BJR} and Proposition 3.1 in \cite{BMMfree}.
%
%                SPEZZAMENTO
%
\begin{pro}\label{spezzGen}
Assume that $G^{ab}$ is finitely generated. Then, with the
notation of Lemma \ref{gene}, the following hold.
\begin{enumerate}
\item[(i)] The map $f_1$ defined to be the restriction 
$f\vert _{\nabla(G)}:\nabla(G)\lto \nabla(G^{ab})$ of 
the projection onto $G^{ab}$, has kernel $N = E(G)\cap
\nabla(G)$. Moreover, $N$ is a central elementary abelian
$2$-subgroup of $[G, G^\varphi]$ of rank at most the $2$-rank $\rk_2(G^{ab})$
of $G^{ab}$.
\item[(ii)] $[G, G^\varphi]/N\simeq \nabla(G^{ab})\times (G\wedge G)$.
\item[(iii)] Suppose either that $G^{ab}$ has no elements of order two or
that $G'$ has a complement in $G$. Then $\nabla(G)\simeq
\nabla(G^{ab})$ and $G\otimes G\simeq \nabla(G)\times (G\wedge
G)$.
\end{enumerate}
\end{pro}
\begin{proof}$\null$\\
(i) Let $w\in \nabla(G)\cap E(G)$. Then $$f_1(w)=f(w)\in
\nabla(G^{ab})\cap E(G^{ab})=1,$$
and so $N\leq \Ker{f_1}$. Conversely, $\Ker{f_1}=\Ker{f}\cap
\nabla(G)=(G'\otimes G)\cap \nabla(G)\leq N$.\\
It is obvious that $N$ is a central subgroup of $G\otimes G$,
since $N$ is contained in $\nabla(G)$. In order to show that $N$ is an
elementary abelian $2$-group, we recall that there is a sequence
of epimorphisms between finitely generated abelian groups
$$\Gamma(G^{ab})\xrightarrow{\,\,\psi\,\,}\nabla(G)\xrightarrow{\,\,f_1\,\,}\nabla(G^{ab}),$$
where $\Gamma(G^{ab})$ is the Whitehead functor on $G$ (see
\cite{BJR}). In particular, if $N_2=\Ker{\psi f_1}$ and
$N_1=\Ker{\psi}$, then $N\simeq N_2/N_1$. From \cite{Whi} (II.
7), we recall some basic facts about the functor $\Gamma$. First, if
$A$ is a finitely generated abelian group such that $A=\prod_i
A_i$, then
$$\Gamma(A)=\prod_i\Gamma(A_i)\times \prod_{i<j}(A_i\otimes A_j).$$
Moreover, $\Gamma$ acts in the following way on cyclic groups:
$$\Gamma(\mm{Z})\simeq \mm{Z}\quad \textrm{ and }\quad
\Gamma(\mm{Z}_n)\simeq
  \begin{cases}
    \mm{Z}_n & \text{if $n$ is odd,} \\
    \mm{Z}_{2n} & \text{if $n$ is even}.$$
  \end{cases}$$
If the $2$-Sylow subgroup of $G^{ab}$ is
$\prod_{i=1}^r\langle x_iG'\rangle$, it follows that $N_2\le
\prod_{i=1}^r\Gamma(\langle x_iG'\rangle)$ and hence that $N_2$ is an elementary
abelian $2$-group whose rank is $r=\rk_2(G^{ab})$. Since
$N\simeq N_2/N_1$, the result follows.\\

\noindent (ii) By Lemma \ref{gene} and (i), we have
$$\frac{G\otimes G}{N}\simeq \frac{\nabla(G)}{N}\times \frac{E(G)}{N}.$$
Note that $\nabla(G)/N\simeq \nabla(G^{ab})$ and
$$E(G)/N\simeq E(G)\nabla(G)/\nabla(G)=(G\otimes G)/\nabla(G)=G\wedge G.$$\\

\noindent (iii) If $G^{ab}$ has no elements of order two, then
$2$ does not divide the order of the torsion part of $\Gamma(G^{ab})$, and so $\Gamma(G^{ab})\simeq
\nabla(G)\simeq \nabla(G^{ab})$, forcing the result.\\
Assume now that $G'$ has a complement $A$ in $G$. If we write
$g\in G$ as $g=xa$, with $x\in G'$ and $a\in A$, by Lemma
3.1 (iv) in \cite{Rocco94} we have that
$$[g,g^\varphi]=[a,a^\varphi],$$
forcing
$$\nabla(G)=\seq{[g,g^\varphi]\vert g\in G}=\seq{[a,a^\varphi]\vert a\in
A}\simeq \nabla(G^{ab}),$$ and $N=1$.
\end{proof}

\noindent {\bf Observation.} In the proof of Proposition
\ref{spezzGen} (i) if $\abs{x_i}=\abs{x_iG'}$, for
$i=1,\ldots,r$, then $N_1$ has rank $r$, so $N\simeq N_2/N_1=1$,
$\nabla(G)\simeq \nabla(G^{ab})$ and $G\otimes G\simeq
\nabla(G)\times (G\wedge G)$.

\medskip
\begin{cor}\label{J2&H2}
Let $G$ be a group such that $G^{\mathrm{ab}}$ is a finitely
generated abelian group with no elements of order two. Then
$J(G)\simeq \Gamma(G^{\mathrm{ab}})\times M(G)$.
\end{cor}
\begin{proof}
Note that $J(G)$ is by definition the kernel of the commutator map
$\kappa:G\otimes G\lto G'$. In particular, $J(G)$ is a central subgroup
of $G$ containing $\nabla(G)$. By Proposition \ref{spezzGen} we
have that $G\otimes G=\nabla(G)\times H$, where $H$ is a subgroup
isomorphic to $G\wedge G$. Therefore, applying Dedekind's
modular law, we have
$$J(G)=\nabla(G)\times (H\cap J(G))\simeq \Gamma(G^{\mathrm{ab}})\times
M(G),$$
since $\nabla(G)\simeq\Gamma(G^{\mathrm{ab}})$ and
$J(G)/\nabla(G)\simeq M(G)$.
\end{proof}

\noindent We recall the notions of \emph{non-abelian tensor
center} $Z^\otimes(G)$ and \emph{non-abelian exterior center} $Z^\wedge(G)$ of a group $G$.
These groups are defined in \cite{Ellis95} as
\begin{align*}
Z^{\otimes}(G)=&\set{g\in G\vert [g,x^\varphi]=1,\, \forall x\in
G}\\
Z^{\wedge}(G)=&\set{g\in G\vert [g,x^\varphi]\in \nabla(G),\,
\forall x\in G}.
\end{align*}
As Ellis showed in \cite{Ellis95} and \cite{Ellis98},
$Z^{\otimes}(G)$ is a characteristic central subgroup of $G$ and is
 the largest normal subgroup $L$ of $G$ such that
$G\otimes G\simeq G/L\otimes G/L$. The non-abelian exterior center $Z^{\wedge}(G)$ is a
central subgroup of $G$ and is equal to the epicenter $Z^*(G)$ of $G$.. In
particular, a group $G$ is capable (that is, is isomorphic to a central quotient $E/Z(E)$
for some group $E$) if and only if $Z^{\wedge}(G)=1$.
\begin{cor}\label{Zcentri}
Let $G$ be any group such that $G^{ab}$ is finitely generated.
With the notation of Proposition \ref{spezzGen}, if $N=1$ then
$Z^{\otimes}(G)=Z^{\wedge}(G)\cap G'$. In particular, the conclusion holds if
$G$ is a finite group of odd order.
\end{cor}
\begin{proof}
By the definition of exterior center we have that
$$[Z^{\wedge}(G)\cap G',G^\varphi]\leq N=1.$$
Therefore $Z^{\wedge}(G)\cap G'\leq Z^{\otimes}(G)$. Conversely,
we trivially have $Z^{\otimes}(G)\leq Z^{\wedge}(G)$. Moreover,
$Z^{\otimes}(G)\leq G'$, as if $x\in Z^{\otimes}(G)$ then
$[G'x,(G'x)^\varphi]$ should be the trivial element of the tensor
product $G^{ab}\otimes G^{ab}$, being the image of
$1=[x,x^\varphi]$ under the natural map from $G\otimes G$ to
$G^{ab}\otimes G^{ab}$. This of course forces $G'x$ being the
identity element of $G^{ab}$, so $x\in G'$.
\end{proof}
\begin{que}
With the notation of Proposition \ref{spezzGen}, is it always true
that $$N=[Z^{\wedge}(G)\cap G',G^\varphi]?$$
\end{que}
Note that a positive answer to the previous question will imply,
by Proposition 9 in \cite{BJR}, that
$$\frac{G\otimes G}{N}\simeq \frac{G}{H}\otimes \frac{G}{H},$$
where $H$ is defined to be $Z^{\wedge}(G)\cap G'$. This is precisely 
the case for generalized quaternion groups and semi-dihedral
groups, $ $ as the reader may check by some easy calculations
and the aid of \cite{BJR}, Proposition 13, and \cite{Ellis95},
Proposition 16.

\section{Structure of the non-abelian exterior square}
We will now describe the structure of the non-abelian exterior
square $G\wedge G$ of a group $G$. 
%Throughout this section we
%prefer to denote the generators of $G\otimes G$ with $g_1\otimes
%g_2$ (instead of $[g_1,g_2^\varphi]$ as previously used). Moreover
%we denote with $g_1\wedge g_2$ the coset of $G\wedge G$ containing $g_1\otimes g_2$.\\
Throughout this section we view the non-abelian tensor square $G\otimes G$ as defined at the beginning of the paper, with generators  $g_1\otimes g_2$, rather than via the isomorphic subgroup $[G,G^\varphi]$ of $\nu(G)$. We denote with $g_1\wedge g_2$ the coset of $G\wedge G$ containing $g_1\otimes g_2$.\\

\noindent Let $G$ be a group and let
$R\xrightarrow{i}F\xrightarrow{\pi}G$ be a presentation for $G$,
where $F$ is a free group. Set $F^\circ$ to be the quotient $F/[F,R]$ and set 
$R^\circ$ to be $R/[F,R]$, so that
\begin{equation}\label{Fequa}
1\lto R^\circ\xrightarrow{i}F^\circ\xrightarrow{\eta}G\lto 1
\end{equation}
  is a central exact
sequence. 
From the sequence (\ref{Fequa}) and by Proposition 7 in
\cite{BJR}, there exists a homomorphism
$$\xi:G\otimes G\lto (F^\circ)'$$ such that $\eta\, \xi$ is the
commutator map $\kappa:G\otimes G\lto G'$. In particular, $\xi$
operates as follows on the generators $g_1\otimes g_2$ of
$G\otimes G$:
$$\xi(g_1\otimes g_2)=[f_1,f_2][F,R],$$
where $f_1$ and $f_2$ are any two preimages of $g_1$ and $g_2$ in
$F$, respectively. Of course, $\xi$ is trivial on the central
subgroup $\nabla(G)$, and so it induces a homomorphism
\begin{equation}\label{XiBar}
\overline{\xi}:G\wedge G\lto (F^\circ)'.
\end{equation}
The following Proposition is the main result of this section. The proof
uses an argument similar to that of Theorem 2 in \cite{Mil}.
\begin{pro}\label{exter}
Let $G$ be a group and let $F$  be a free group such that $G\simeq
F/R$ for some normal subgroup $R$ of $F$. Then
$$G\wedge G\simeq\frac{F'}{[F,R]}.$$
\end{pro}
\begin{proof}
We will show that the map $\overline{\xi}$ defined in (\ref{XiBar}) is an
isomorphism.\\
The surjectivity of $\overline{\xi}$ is immediate. Let $\pi:F\to
G$ be the projection with kernel $R$. An arbitrary generator
$[f_1,f_2][F,R]$ of $F'/[F,R]$ lies in the image of
$\overline{\xi}$,
since $\overline{\xi}(\pi(f_1)\wedge\pi(f_2))=[f_1,f_2][F,R]$.\\
We now prove that $\overline{\xi}$ is injective.
Using the same notation as in the introduction, for any group $X$
we set $J(X)$ to be $\Ker{\kappa}$ and $M(X)$ to be $\Ker{\kappa'}$, 
where $\kappa$ and $\kappa'$ are
the commutator maps  $\kappa:X\otimes X\lto X'$ and
$\kappa':X\wedge X\lto X'$ respectively, so that $J(X)/\nabla(X)\simeq M(X)$.\\
Let $\phi$ be the map $\overline{\xi}$ restricted to $M(G)$. We
want to show that $\phi$ is injective. Note that $\phi$ is a map 
$$\phi:M(G)\lto R^\circ\cap (F^\circ)'=\frac{F'\cap R}{[F,R]}.$$
Now the quotient map $\eta:F^\circ\lto G$ induces a homomorphism
$$\eta_{\ast}:M(F^\circ)\lto M(G)$$
(which is the restriction to $M(F^\circ)$ of the map sending
$f_1[F,R]\wedge f_2[F,R]$ to $\eta(f_1)\wedge \eta(f_2)$). It is
easy to notice that the following is an exact sequence
$$M(F^\circ)\xrightarrow{\eta_\ast}M(G)\xrightarrow{\phi}R^\circ\cap (F^\circ)'.$$
Therefore, in order to show that $\phi$ is injective, we will
prove that $\eta_\ast$ is the trivial map. If
$\alpha:J(F^\circ)\lto M(F^\circ)$ is the quotient map, we show
that
$\eta_\ast(\alpha(J(F^\circ))=\eta_\ast(\alpha(\nabla(F^\circ))$.\\
Let $\overline{w}=\prod(\overline{x_i}\otimes \overline{y_i})\in
J(F^{\circ})$, with $\overline{x_i},\overline{y_i}\in F^\circ$.
Then there exist $x_i,y_i\in F$ and $w=\prod(x_i\otimes y_i)\in
F\otimes F$ such that $\lambda_{\otimes}(w)=\overline{w}$ and
$\lambda(\kappa(w))=\kappa(\overline{w})=1$ (here
$\lambda_{\otimes}:F\otimes F\lto F^\circ\otimes F^\circ$ is the
map induced by the projection $\lambda:F\lto F^\circ$). Thus
$\kappa(w)=\prod[x_i,y_i]\in \Ker \lambda=[F,R]$. As $F$ is a free
group its Schur multiplicator is trivial, so $M(F)=1$, that is,
$J(F)=\nabla(F)$. In particular, modulo $\nabla(F)$, the product $w$ is
equivalent to $\prod(f_j\otimes r_j)$, for some $f_j\in F$ and
$r_j\in R$. So $w=\prod(f_j\otimes r_j)z$, for some $z\in
\nabla(F)$. We note that $(\alpha(\lambda_{\otimes}(z))=1$. Now
\begin{align*}
\eta_\ast(\alpha(\overline{w}))&=\eta_\ast(\alpha(\lambda_{\otimes}(w)))=
\eta_\ast(\alpha(\lambda_{\otimes}(\prod(f_j\otimes r_j)z)))\\
                          &= \prod(\eta_\ast \alpha(\overline{f_j}\otimes 1))= 1,
\end{align*}
therefore $\eta_\ast$ is the trivial map and $\phi$ is injective.\\
Finally in order to show that $\overline{\xi}$ is an isomorphism
between $G\wedge G$ and $F'/[F,R]$, we apply the Short Five Lemma
(\cite{AtMcD}, Proposition 2.10) to the following commutative
diagram.
%\begin{align*}
%1\lto &M(G) \xrightarrow{\,\,i\,\,}G\wedge G\xrightarrow{\,\,\kappa'\,\,}\,G'\lto 1\\
%      &\quad \downarrow{\small{\phi}}\quad\quad\quad\downarrow{\overline{\xi}}\quad\quad\quad\downarrow{1_{G'}}  \\
%1\lto&\frac{R\cap F'}{[F,R]}\xrightarrow{\,\,\,
%i\,\,\,}\frac{F'}{[F,R]}\,\xrightarrow{\,\,\,\eta\,\,\,}\,G'\lto 1
%\end{align*}
\begin{equation*}
\begin{CD}
1@>>>M(G)@>i>>G\wedge G@>\kappa'>>G'@>>>1\\
@. @VV\phi V @VV\overline{\xi}V @VV1_{G'}V\\
1@>>>\frac{R \cap F'}{[F,R]}@>i>> \frac{F'}{[F,R]} @>\eta>> G @>>> 1\\
\end{CD}
\end{equation*}
As both $\phi$ and the identity map of $G'$ are injective, it
follows that also $\overline{\xi}$ is injective. This concludes
the proof that $\overline{\xi}$ is an isomorphism.
\end{proof}
As consequences of the results above we now describe the
structures of the non-abelian tensor squares of some groups that are ``universal''
 in the sense that they are free in suitable varieties.

\begin{cor}[\cite{BJR}, Proposition 6]\label{freegroups}
Let $F_n$ be a free group of rank $n$. Then
$$F_n\otimes F_n\simeq \mm{Z}^{n(n+1)/2}\times (F_n)'.$$
\end{cor}
\begin{proof}
Since $F_n^{ab}$ is a free abelian group of rank $n$, by Proposition
\ref{spezzGen} we have $\nabla(F_n)\simeq \nabla(F_n^{ab})\simeq
\mm{Z}^{n(n+1)/2}$ and $F_n\otimes F_n\simeq
\nabla(F_n^{ab})\times F_n\wedge F_n$. Finally
Proposition~\ref{exter} (with $F=F_n$ and $R=1$) gives $F_n\wedge F_n \simeq (F_n)'$.
\end{proof}
Proposition \ref{exter} gives a nice description in the case of
free nilpotent groups of finite rank. For a more detailed
description of this case we refer to the paper \cite{BMMfree}.

\begin{cor}[\cite{BMMfree}, Corollary 1.7]\label{freenilpotent}
Let $G=\mathcal{N}_{n,c}$ be the free nilpotent group of rank
$n>1$ and class $c\geq 1$. Then
$$G\otimes G\simeq \mm{Z}^{n(n+1)/2}\times (\mathcal{N}_{n,c+1}) '.$$
\end{cor}
\begin{proof}
As before, $G^{ab}$ is a free abelian group of rank $n$, and so we
have $\nabla(G)\simeq \nabla(G^{ab})\times \mm{Z}^{n(n+1)/2}$ and
$G\otimes G\simeq \nabla(G^{ab})\times G\wedge G$. Finally apply
Proposition~\ref{exter} with $F$ free group of rank $n$ and
$R=\gamma_{c+1}(F)$.
\end{proof}
\begin{cor}\label{freesolvable}
Let $F$ be the free group of finite rank $n>1$, let $d$ be a natural
number, and let $G=F/F^{(d)}$ be the free solvable
group $\mathcal{S}_{n,d}$ of derived length $d$ and rank $n>1$. Then
$$G\otimes G\simeq \mm{Z}^{n(n+1)/2}\times F'/[F,F^{(d)}]$$
is an extension of a nilpotent group of class $\leq 3$ by a free
solvable group of derived length $d-2$ and infinite rank. In
particular, if $d=2$, then $G\otimes G$ is a nilpotent group.
\end{cor}
\begin{proof}
Once again, Propositions \ref{spezzGen} and \ref{exter} imply that
$G\otimes G$ has the described factorization. Note that $F^{(d-1)}/[F,F^{(d)}]$ 
is a normal subgroup of the group
$F'/[F,F^{(d)}]$ 
and that $F^{(d-1)}/[F,F^{(d)}]$ is nilpotent of class at most $3$, as it is a
quotient of $F^{(d-1)}/\gamma_3(F^{(d-1)})$. So $M=
\mm{Z}^{n(n+1)/2}\times F^{(d-1)}/[F,F^{(d)}]$ is also nilpotent
of class at most $3$ and $G\otimes G/M$ is
isomorphic to $F'/F^{(d-1)}$, so it is free solvable of derived
length $d-2$. The fact that $F'/F^{(d-1)}$ is of infinite rank
follows from the well-known fact that $F'$ is not finitely
generated.
\end{proof}
We recall, in view of Theorem A in \cite{Bau}, that the Schur
multiplicator of $\mathcal{S}_{n,d}$ is not finitely generated. In
particular, $\mathcal{S}_{n,d}\wedge \mathcal{S}_{n,d}$ can also be
viewed as an extension of a central abelian group of infinite rank
by a free solvable group of derived length $d-1$.
\medskip

We end this section by applying our results to a particular finite $p$-group.\\
Let $d$ be an integer and, as before, denote by $F_d$ the free
group on $d$ generators. We recall that for every integer $i$ the
group $\gamma_i(F_d)/\gamma_{i+1}(F_d)$ is free abelian of rank
$$m_d(i):=\frac{1}{i}\sum_{t\vert i}\mu(t)d^{i/t}, $$
where $\mu$ is the Mobius function (see \cite{Karpi} Chapter 3.2).\\
We also recall for a fixed prime number $p$ the notion of \emph{lower central $p$-series}
 of a group $G$. The terms
of this series are $\set{\lambda_i(G)}_{i\geq 1}$, where 
\begin{align*}
\lambda_1(G)&=G\\
\lambda_{k+1}(G)&=[\lambda_k(G),G]\lambda_k(G)^p, \textrm{ for
} k\geq 1.
\end{align*}
We note that this series is the most rapidly descending central
series of $G$ whose factors have exponent $p$ (see \cite{Karpi},
Chapter 3). The lower central $p$-series will be used in the next
section to find some bounds on the orders of the non-abelian
tensor and exterior squares of finite $p$-groups.  Now we
exhibit an explicit calculation of these objects in a particular
case.\\  For every pair of positive integers $d$ and $c$ define
$G_{d,c}$ to be the quotient $F_d/\lambda_{c+1}(F_d)$. According to \cite{Karpi}
(Theorem 3.2.10), $G_{d,c}$ is a finite $p$-group of class $c$ and
order $p^m$, where $m=\sum_{j=1}^c(c+1-j)m_d(j)$.

\begin{cor}\label{lambdagroups} With the above notation, we have that
$G_{d,c}\wedge G_{d,c}\simeq (G_{d,c+1})'$ and
$$G_{d,c}\otimes G_{d,c}\simeq (\mathbb{Z}_{p^{c}})^{d(d+1)/2}\times (G_{d,c+1})'.$$
\end{cor}
\begin{proof}
Let $G:=G_{d,c}$. We first prove that
\begin{equation}\label{tsdp}
G\otimes G\simeq \nabla(G)\times (G\wedge G).
\end{equation}
For $p$ odd (\ref{tsdp}) follows from  Proposition \ref{spezzGen}, while
for the case $p=2$ a little more care is needed. More precisely,
we observe that if $F_d=\langle f_1,\ldots,f_d\rangle$, then the
image $x_i$ in $G=F_d/\lambda_{c+1}(F_d)$ of the generator
$f_i$ of $F_d$ has order $p^c$ for each $i=1,\ldots, d$.
Moreover, by Theorem 3.2.10 in \cite{Karpi}, $G^{ab}$ is
isomorphic to a direct product of $d=m_d(1)$  cyclic groups
$\mm{Z}_{p^c}$ of order $p^c$. So now our result follows from
the observation following Proposition \ref{spezzGen}.

We have that $\nabla(G)\simeq (\mm{Z}_{p^{c}})^{d(d+1)/2}$. We
claim that the derived subgroup of a covering group for $G$ is
isomorphic to $(G_{d,c+1})'$. In the following, let
$L_i$ denote $\lambda_i(F_d)$, $i\geq 1$. We note that the group
$G_{d,c+1}=F_d/L_{c+2}$ has $L_{c+1}/L_{c+2}$ as a
central elementary abelian subgroup. Moreover, the subgroup
$$\frac{M}{L_{c+2}} \mbox{\ defined to be\ }(G_{d,c+1})'\cap\frac{L_{c+1}}{L_{c+2}}=
\frac{\gamma_2(F_d)L_{c+2}\cap L_{c+1}}{L_{c+2}}$$
%
%=\frac{L_{c+2}(\gamma_2(F_d)\cap
%L_{c+1})}{L_{c+2}}
%
is isomorphic to
$$\frac{L_{c+1}\cap
\gamma_2(F_d)}{L_{c+2}\cap\gamma_2(F_d)},$$
which is isomorphic to $M(G)$, by Theorem 3.2.10 in \cite{Karpi}.
Now let $H/L_{c+2}$ be a complement of $M/L_{c+2}$ in
$L_{c+1}/L_{c+2}$ and  consider the factor group
$$\overline{G}_{d,c+1}=\frac{G_{d,c+1}}{H/L_{c+2}}.$$
If $N\le {G}_{d,c+1}$ we denote with $\overline N$ the image of $N$ in $\overline{G}_{d,c+1}$ under the canonical projection; it follows that
$$M(G)\simeq\overline{M}\leq Z(\overline{G}_{d,c+1})\cap
(\overline{G}_{d,c+1})'.$$
Moreover, $\overline{G}_{d,c+1}/\overline{M}\simeq F_d/L_{c+1}=G$,
so $\overline{G}_{d,c+1}$ is a covering group for $G$. Finally,
note that
$$(\overline{G}_{d,c+1})'=\frac{(F_d)'H}{H}\simeq \frac{(F_d)'}{(F_d)'\cap
H}=\frac{(F_d)'}{(F_d)'\cap L_{c+2}}=(G_{d,c+1})'.$$
\end{proof}
%
%     p gruppi
%
\section{Non-abelian tensor squares of finite $p$-groups}
Throughout this section $G$ is a finite $p$-group, for some
prime $p$. We start with a lemma concerning the lower central
$p$-series of $G$. 
 We again identify the group $G\otimes G$ with its
isomorphic image $[G,G^\varphi]$ in the group $\nu(G)$ defined in
Section 2.

\begin{lem}\label{frat}
Let $G$ be a finite $p$-group. Then for every $k\geq 1$,
$$[\lambda_k(G),G^\varphi]=[G,(\lambda_k(G))^\varphi].$$
\end{lem}
\begin{proof}
We prove the result by induction on $k$. Since the result is trivial for
$k=1$, we assume $[\lambda_k(G),G^\varphi]=[G,(\lambda_k(G))^\varphi]$ 
 and show that
$[\lambda_{k+1}(G),G^\varphi]=[G,(\lambda_{k+1}(G))^\varphi]$.\\
First note that, since $[\lambda_k(G),G,G^\varphi]$ and
$[\lambda_k(G)^p,G^\varphi]$ are both normal in $\nu(G)$, we have,
using $[xy,a^\varphi]=\null^{x}[y,a^\varphi][x,a^\varphi]$ with
$x\in [\lambda_k(G),G]$, $y\in \lambda_k(G)^p$ and $a\in G$, that
$$[\lambda_{k+1}(G),G^\varphi]=[\lambda_k(G),G,G^\varphi][\lambda_k(G)^p,G^\varphi].$$
(As $[xy,a^\varphi]=\null^{x}[y,a^\varphi][x,a^\varphi]$, with
$x\in [\lambda_k(G),G]$, $y\in \lambda_k(G)^p$ and $a\in G$).\\
Using Lemma \ref{basiccomm}(ii), we have that
$$[\lambda_k(G),G,G^\varphi]=[\lambda_k(G)^\varphi,G^\varphi,G]=
[G,[\lambda_k(G),G]^\varphi]\leq [G,\lambda_{k+1}(G)^\varphi].$$
Thus our proof will be complete if we show that
$[\lambda_k(G)^p,G^\varphi]\leq
[G,\lambda_{k+1}(G)^\varphi]$.\\
Define $R$ to be $[\lambda_k(G),G,G^\varphi]$
($=[G,[\lambda_k(G),G]^\varphi]$).\\
Note that $R$ contains the derived subgroup of
$[\lambda_k(G),G^\varphi]$. To see this, we observe that
$[\lambda_k(G),G^\varphi]'$ is generated by the elements
$$[[x,a^\varphi],[y,b^\varphi]],\quad \textrm{ where }x,y\in\lambda_k(G)
\textrm{ and }a^\varphi,b^\varphi\in G^\varphi,$$
and, by Lemma \ref{basiccomm}(i) and the defining properties of
$\nu(G)$, we have that
$[[x,a^\varphi],[y,b^\varphi]]=[[x,a],[y,b]^\varphi]\in R$.\\
We claim that the following hold:
\begin{equation}\label{uno}
[x^m,a^\varphi]\in [x,a^\varphi]^mR\quad \textrm{ for all }
x\in \lambda_k(G), a^\varphi\in G^\varphi, m\in \mm{N}
\end{equation}
\begin{equation}\label{due}
[y,(b^m)^\varphi]\in [y,b^\varphi]^mR\quad \textrm{ for all }
y\in G, b^\varphi\in (\lambda_k(G))^\varphi,  m\in \mm{N}.
\end{equation}
We prove (\ref{uno}) by induction on $m$. The proof of
(\ref{due}) is similar.\\
If $m=1$ then (\ref{uno}) is trivially true. Let $m\geq 2$. Then
\begin{align*}
[x^m,a^\varphi]&=[x\cdot
x^{m-1},a^\varphi]=\null^x[x^{m-1},a^\varphi][x,a^{\varphi}]\\
&=[x^{m-1},(\null^xa)^\varphi][x,a^{\varphi}].
\end{align*}
Now the claim is proved since, by induction on $m$, the term
$[x^{m-1},(\null^xa)^\varphi]$ lies in the coset
\begin{align*}
[x,(\null^xa)^\varphi]^{m-1}R&=[x,[x,a]^\varphi
a^\varphi]^{m-1}R=\\
&=([x,[x,a]^\varphi]\cdot\null^{[x,a]^\varphi}[x,a^\varphi])^{m-1}R=\\
&=([x,[x,a]^\varphi]\cdot[x,a^\varphi])^{m-1}R=([x,a^\varphi])^{m-1}R,
\end{align*}
by a repeated use of Lemma \ref{basiccomm}, and the fact that
$[x,[x,a]^\varphi]=[x,a,x^\varphi]^{-1}\in R$. Therefore our
claims (\ref{uno}) and (\ref{due}) are true, and we now complete
the proof of the lemma as follows. We have that
$[\lambda_k(G)^p,G^\varphi]$ is generated by elements of the form
$[x^p,a^\varphi]$ with $x\in \lambda_k(G)$ and $a^\varphi\in
G^\varphi$. By (\ref{uno}), $[x^p,a^\varphi]\in
([x,a^\varphi])^pR$. Now $[x,a^\varphi]\in
[\lambda_k(G),G^\varphi]=[G,(\lambda_k(G))^\varphi]$ by the
inductive hypothesis, so we may write
$$[x,a^\varphi]=w_1\cdot\ldots\cdot w_l,$$
where $w_i=[y_i,b_i^\varphi]$, $y_i\in G$ and $b_i^\varphi\in
\lambda_k(G)^\varphi$ for $i=1,\ldots,l$. In particular, since
$[\lambda_k(G),G^\varphi]/R$ is abelian we have that
$([x,a^\varphi])^pR=w_1^p\ldots w_l^pR.$
Finally, by (\ref{due}) $w_i^pR=[y_i,(b_i^p)^\varphi]R$ for all
$i=1,\ldots,l$, forcing
$$[x^p,a^\varphi]\in
R[G,(\lambda_k(G)^p)^\varphi]=[G,(\lambda_{k+1}(G))^\varphi].$$
\end{proof}
The following result is an improvement of Corollary 3.12 in
\cite{Rocco91}. In his PhD thesis A. McDermott proves this result using
arguments different from ours (see \cite{EllisMcD}).
\begin{pro}\label{order} Let $G$ be a finite group of order $p^n$ ($p$ a prime)
and let $d=d(G)$ be the minimum number of generators of $G$. Then
$p^{d^2}\leq \abs{[G,G^\varphi]}\leq p^{nd}$.
\end{pro}
\begin{proof}
Of course $\abs{[G,G^\varphi]}\geq p^{d^2}$, as $G\otimes G$
admits  $G/\Phi(G)\otimes G/\Phi(G)$ as a quotient, and $G/\Phi(G)\otimes G/\Phi(G)$
 is elementary abelian of order $p^{d^2}$, since it is an
ordinary tensor product.\\
Let $\lambda_k(G)$ be the last non-trivial term of the series
$\set{\lambda_i(G)}_i$, and let $\pi:G\lto
\overline{G}:=G/\lambda_k(G)$ be the quotient map. $\pi$ induces a
natural epimorphism, say, $\widetilde{\pi}:[G,G^\varphi]\lto
[\overline{G},\overline{G}^\varphi]$. According to \cite[Remark
3]{Rocco91} and using the previous Lemma, the kernel of
$\widetilde{\pi}$ consists in the subgroup
$$\Ker{\widetilde{\pi}}=[\lambda_k(G),G^\varphi][G,\lambda_k(G)^\varphi]
=[\lambda_k(G),G^\varphi].$$
Since $\lambda_k(G)$ is a central elementary abelian subgroup of
$G$, by Lemma \ref{basiccomm}(ii), we have that
$\Ker{\widetilde{\pi}}$ is an elementary abelian $p$-subgroup
lying in the center of $\nu(G)$. Thus the map
\begin{align*}
\theta:\lambda_k(G)\times G&\lto [\lambda_k(G),G^\varphi]\\
(a,g)&\lmto [a,g^\varphi],
\end{align*}
is bilinear. Let $\lambda_k(G)$ be generated by the set
$\set{a_i\vert i=1,\ldots, d_k}$ and let $G$ be generated by
$\set{g_i\vert i=1,\ldots,d}$. Therefore $\Ker{\widetilde{\pi}}$
is generated by the set
$$\set{[a_i,g_j^\varphi]\vert i=1,\ldots,d_k,\, j=1,\ldots, d},$$
forcing $\abs{\Ker{\widetilde{\pi}}}\leq p^{d\cdot d_k}$, and
$\abs{[G,G^\varphi]}\leq p^{d\cdot
d_k}\abs{[\overline{G},\overline{G}^\varphi]}$. By induction we
obtain that
$$\abs{[G,G^\varphi]}\leq p^{d\cdot d_k}\cdot\ldots\cdot
p^{d^2}=p^{d\sum_{i=1}^{k}d_i}=p^{nd}.$$
\end{proof}
\begin{rem}
Homocyclic abelian groups show that the upper bound in the
Proposition \ref{order} is best possible. Another example in which
the upper bound is reached is when $G$ is the group
$G_{2,2}:=F_2/\lambda_3(F_2)$.
\end{rem}
As a consequence of our results we have the following bound for
the order of the Schur multiplicator of finite $p$-groups.
\begin{cor}\label{boundSchur}
Let $G$ be a finite $p$-group of order $p^n$ with $d=d(G)$
generators. If $p$ is odd, the order of the Schur multiplicator
$M(G)$ of $G$ is at most $p^{d(n-(d+1)/2)}$. If $p=2$, then
$\abs{M(G)}\leq 2^{d(n-(d+3)/2)}$.
\end{cor}
\begin{proof}
By Proposition \ref{exter} and the definition of the exterior
square
$$\abs{M(G)}\abs{G'}=\abs{G\wedge G}=\frac{\abs{G\otimes
G}}{\abs{\nabla(G)}}.$$
If $p$ is odd, by Proposition \ref{spezzGen}, $\nabla(G)\simeq
\nabla(G^{ab})$, and so $\abs{\nabla(G)}\geq p^{d(d+1)/2}$.
If $p=2$, then $\abs{\nabla(G)}\geq p^{d(d+3)/2}$. The
proof is now completed by using the bounds given in Proposition
\ref{order}.
\end{proof}

%----------------------------------------------------------

%------------------------------------------------------------
\medskip

{\small\sc Department of Mathematics and Computer Science, Saint Louis
  University, St. Louis, MO 63103 USA}, email: {\it blythrd@slu.edu}

\medskip
{{\small\sc Dipartimento di Matematica "Ulisse Dini"', Universit\`a di
  Firenze, Viale Morgagni, 67/a, 50134 Firenze, Italy},  email: {\it
  fumagal@math.unifi.it} 

\medskip

{{\small\sc Dipartimento di Matematica, Universit\`a di Bologna, Piazza
di Porta San Donato 5, 40127 Bologna, Italy},  email: {\it mmorigi@dm.unibo.it}}

\end{document}